\pdfoutput=1
\documentclass[11pt]{amsart}

\usepackage[utf8]{inputenc}
\usepackage{amsmath,amssymb,amsthm,mathtools}
\usepackage{booktabs}
\usepackage{placeins}
\usepackage[margin=1.15in]{geometry}
\usepackage[hidelinks]{hyperref}

\newtheorem{theorem}{Theorem}[section]
\newtheorem{proposition}[theorem]{Proposition}
\newtheorem{lemma}[theorem]{Lemma}

\theoremstyle{definition}
\newtheorem{definition}[theorem]{Definition}
\newtheorem{hypothesis}[theorem]{Hypothesis}
\theoremstyle{remark}
\newtheorem{remark}[theorem]{Remark}

\numberwithin{equation}{section}

\newcommand{\PP}{\mathbb{P}}
\newcommand{\EE}{\mathbb{E}}
\newcommand{\fS}{\mathfrak{S}}
\newcommand{\lpf}{P^{-}}
\newcommand{\eps}{\varepsilon}

\begin{document}

\title{Multiplicity constants for rough numbers in prime gaps}

\author{Bahar U\u{g}urdo\u{g}an}
\address{Independent Researcher, Istanbul, T\"urkiye}
\thanks{The author received her Ph.D.\ in Electrical and Electronics Engineering from Yeditepe University, Istanbul, T\"urkiye.}


\begin{abstract}
Let $p_n$ be the $n$th prime and $g_n=p_{n+1}-p_n$. Following Gafni and
Tao, an integer $m$ with $p_n<m<p_{n+1}$ is \emph{rough} for the gap if
its least prime factor is at least $g_n$; let $M_n$ be the number of such
integers. Gafni and Tao proved that the number of $n$ with
$X\le p_n\le 2X$ and $M_n=0$ is $O(X/\log^2X)$, and that it is
asymptotic to $cX/\log^2X$ under a form of the Dickson--Hardy--Littlewood
prime tuples conjecture. We study $N_r(X)$, the number of $n$ with
$X\le p_n\le 2X$ and $M_n<r$. For every $r\ge1$ we define a constant
$c_r$ by an endpoint-conditioned local sieve model, with $c_1=c$, and
we prove unconditionally that
\[
c_r=h_r^*+O_\eps\bigl(r^{1/2}(\log r)^{1+\eps}\bigr)\qquad(r\ge3),
\]
where $h_r^*$ is the least even $h\ge2$ with $h\prod_{p<h}(1-1/p)\ge r$.
Consequently $c_r=H_r\bigl(1+O(e^{-\kappa\sqrt{\log r}})\bigr)$ for an
absolute constant $\kappa>0$, where $H_r>e$ solves
$e^{-\gamma}H_r/\log H_r=r$ ($\gamma$ is Euler's constant), and
$c_r=e^\gamma r\bigl(\log r+\log\log r+\gamma+
\frac{\log\log r+\gamma}{\log r}
+O\bigl(\frac{(\log\log r)^2}{(\log r)^2}\bigr)\bigr)$.
The proof combines the Montgomery--Vaughan moment bounds for reduced
residues in short intervals with an exact removal of the endpoint
conditioning. For each fixed $r$, the argument of Gafni and Tao extends
to give $N_r(X)\ll_r X/\log^2X$ unconditionally, and
$N_r(X)\sim c_rX/\log^2X$ under their conjecture. Exact enumeration of the
local model gives $c\ge2.7413$.
\end{abstract}

\maketitle

\section{Introduction}\label{sec:intro}

Let $p_n$ denote the $n$th prime, let $g_n:=p_{n+1}-p_n$, and for an
integer $m>1$ let $\lpf(m)$ denote its least prime factor. Following
Gafni and Tao~\cite{GT}, an integer $m$ with $p_n<m<p_{n+1}$ is called
\emph{rough for the gap} $(p_n,p_{n+1})$ if
\[
\lpf(m)\ge g_n .
\]
Erd\H{o}s~\cite[p.~74]{Er79} expected that almost every prime gap contains
such an integer (see also~\cite{Bl}). Gafni and Tao
\cite[Theorem~1.1]{GT} proved that, if $N(X)$ denotes the number of $n$
with $p_n\in[X,2X]$ for which $(p_n,p_{n+1})$ contains no rough number,
then
\[
N(X)\ll\frac{X}{\log^2X},
\]
and that, assuming the Dickson--Hardy--Littlewood prime tuples conjecture
in the form \cite[Conjecture~4.1]{GT},
\[
N(X)\sim c\,\frac{X}{\log^2X}
\]
for an explicitly describable constant $c>0$, which they expect to lie
between $2.7$ and $2.8$.

In this paper we consider the number of rough integers in a prime gap.
Put
\[
M_n:=\#\{m:\ p_n<m<p_{n+1},\ \lpf(m)\ge g_n\},
\]
and, for integers $r\ge1$,
\begin{gather*}
N_r(X):=\#\{n:\ X\le p_n\le 2X,\ M_n<r\},\\
N_{h,r}(X):=\#\{n:\ X\le p_n\le 2X,\ g_n=h,\ M_n<r\}.
\end{gather*}
Thus $N_1(X)=N(X)$.

\subsection{The anchored local model}\label{subsec:model}

Throughout, unless stated otherwise, $p$ denotes a prime, $r$ a positive
integer, and $h\ge2$ an even integer. Put
\[
P_h:=\prod_{p<h}p,\qquad
V_h:=\prod_{p<h}\Bigl(1-\frac1p\Bigr)=\frac{\varphi(P_h)}{P_h}
\]
(so $P_2=1$ and $V_2=1$), and
\[
A_h:=\{\,b\bmod P_h:\ (b(b+h),P_h)=1\,\}.
\]
For $b\in\mathbb Z$ let
\[
\rho_h(b):=\#\{\,1\le i\le h-1:\ (b+i,P_h)=1\,\},
\]
which depends only on $b\bmod P_h$. We write $\PP_{A_h}$ for the uniform
probability measure on $A_h$ and put
\[
q_{h,r}:=\PP_{A_h}(\rho_h<r).
\]
For a finite set $\mathcal H\subset\mathbb Z$ let $\nu_p(\mathcal H)$ be
the number of residue classes modulo $p$ occupied by $\mathcal H$, and let
\[
\fS(\mathcal H):=\prod_p\Bigl(1-\frac{\nu_p(\mathcal H)}p\Bigr)
\Bigl(1-\frac1p\Bigr)^{-|\mathcal H|}
\]
be the Hardy--Littlewood singular series~\cite{HL}. For even $h$ we put
\[
\Sigma(h):=\fS(\{0,h\})=\fS_2\prod_{\substack{p\mid h\\ p>2}}\frac{p-1}{p-2},
\qquad
\fS_2:=2\prod_{p>2}\Bigl(1-\frac1{(p-1)^2}\Bigr)=1.3203236\ldots,
\]
so that $\fS_2$ is the twin prime constant (denoted $\fS$ in~\cite{GT}).

\begin{definition}\label{def:constants}
For even $h\ge2$ and $r\ge1$ let
\[
c_{h,r}:=\Sigma(h)\,q_{h,r},
\qquad
c_r:=\sum_{\substack{h\ge2\\ h\ \mathrm{even}}}c_{h,r}.
\]
\end{definition}

The terms are nonnegative, and $c_r<\infty$ by Lemma~\ref{lem:convergence}
below. Moreover $c_r>0$: since $P_4=6$, one has $A_4=\{1\bmod 6\}$ and
$\rho_4(1)=0$ (each of $2,3,4$ is divisible by $2$ or $3$), so $q_{4,r}=1$
and $c_r\ge c_{4,r}=\Sigma(4)=\fS_2$ for every $r\ge1$. Note also that
$\rho_2\equiv1$ (as $P_2=1$), so $c_{2,1}=0$ and $c_{2,r}=\fS_2$ for
$r\ge2$; this reflects the fact that $p_n+1$ is rough for every twin prime
gap.

The model has the following exact prime-gap interpretation.

\begin{lemma}\label{lem:interpretation}
If $g_n=h$ and $p_n>h$, then $p_n\bmod P_h\in A_h$ and
$M_n=\rho_h(p_n)$.
\end{lemma}

\begin{proof}
The primes $p_n$ and $p_{n+1}=p_n+h$ exceed $h$, hence are coprime to
$P_h$. For $1\le i\le h-1$ one has $\lpf(p_n+i)\ge h$ if and only if no
prime $p<h$ divides $p_n+i$, that is, if and only if $(p_n+i,P_h)=1$.
\end{proof}

\begin{remark}[Comparison with~\cite{GT}]\label{rem:GT}
Let $h\ge4$ be even. By the Chinese remainder theorem,
\[
|A_h|=\prod_{p<h}\bigl(p-\nu_p(\{0,h\})\bigr)
=\prod_{\substack{2<p<h\\ p\nmid h}}(p-2)\prod_{\substack{2<p<h\\ p\mid h}}(p-1),
\]
and every odd prime divisor of $h$ is smaller than $h$. Comparing with
the formula for $\Sigma(h)$ we obtain
\begin{equation}\label{eq:Sigma-over-A}
\frac{\Sigma(h)}{|A_h|}=\frac{\fS_2}{\prod_{2<p<h}(p-2)},
\qquad\text{hence}\qquad
c_{h,1}=\frac{\fS_2\,|\Omega_h|}{\prod_{2<p<h}(p-2)},
\end{equation}
where $\Omega_h:=\{b\in A_h:\ \rho_h(b)=0\}$. This is exactly the
constant $c_h$ of \cite[Section~4]{GT}, and $c_{2,1}=0$ corresponds to the
vanishing of the gap-$2$ term there. Consequently $c_1=c$. More
generally, $c_{h,r}$ is the constant predicted for $N_{h,r}(X)$; see
Proposition~\ref{prop:fixed-h}.
\end{remark}

\subsection{Main results}\label{subsec:results}

For $r\ge3$ let
\[
h_r^*:=\min\{\,h\ge2\ \text{even}:\ hV_h\ge r\,\}.
\]
Since $2V_2=2<3$, the minimum is attained at some $h\ge4$ (in fact
$h_3^*=16$). For $r\ge2$ let $H_r>e$ denote the unique solution of
\[
e^{-\gamma}\,\frac{H_r}{\log H_r}=r ;
\]
it exists and is unique because $x\mapsto x/\log x$ increases from $e$ to
$\infty$ on $(e,\infty)$ and $e^{1-\gamma}<2$. Here $\gamma$ is Euler's
constant.

\begin{theorem}\label{thm:main}
\leavevmode
\begin{enumerate}
\item[(a)] For every fixed $\eps>0$ and all integers $r\ge3$,
\begin{equation}\label{eq:main-a}
c_r=h_r^*+O_\eps\bigl(r^{1/2}(\log r)^{1+\eps}\bigr).
\end{equation}
\item[(b)] There is an absolute constant $\kappa>0$ such that, for all
integers $r\ge3$,
\begin{equation}\label{eq:main-b}
c_r=H_r\bigl(1+O(e^{-\kappa\sqrt{\log r}})\bigr).
\end{equation}
\item[(c)] As $r\to\infty$,
\begin{equation}\label{eq:main-c}
c_r=e^\gamma r\Bigl(\log r+\log\log r+\gamma
+\frac{\log\log r+\gamma}{\log r}
+O\Bigl(\frac{(\log\log r)^2}{(\log r)^2}\Bigr)\Bigr).
\end{equation}
In particular $c_r\sim e^\gamma r\log r$.
\end{enumerate}
\end{theorem}

\begin{theorem}\label{thm:gaps}
Let $r\ge1$ be fixed.
\begin{enumerate}
\item[(a)] For $X\ge3$ one has $N_r(X)\ll_r X/\log^2X$.
\item[(b)] Assume Hypothesis~\ref{hyp:DHL} below (which is the form of
\cite[Conjecture~4.1]{GT} used in \cite[Section~4]{GT}). Then
\[
N_r(X)\sim c_r\,\frac{X}{\log^2X}\qquad(X\to\infty).
\]
\end{enumerate}
\end{theorem}

For $r=1$, Theorem~\ref{thm:gaps} is \cite[Theorem~1.1]{GT}, and the
argument below is a direct extension of the one in~\cite{GT}; the
modifications needed for $r\ge2$ are the threshold $R'\ge2r$ in the
moment arguments and the removal of prime triples in the conditional
part (in Lemma~\ref{lem:short} we also use the local bound~\eqref{eq:MV}
in place of the sixth-moment estimate of \cite[proof of Lemma~4.1]{GT}).
Our main concern is the quantitative dependence of the local constants
on~$r$ in Theorem~\ref{thm:main}.

\begin{remark}\label{rem:scope}
Theorem~\ref{thm:main} concerns the deterministic constants $c_r$ as
$r\to\infty$, whereas Theorem~\ref{thm:gaps} concerns $N_r(X)$ for each
fixed~$r$. No asymptotic formula for $N_{r(X)}(X)$ with a growing
threshold $r=r(X)$ is claimed here; this would require prime-pair input
that is uniform in the modulus $P_h$ and in~$h$.
\end{remark}

\begin{remark}\label{rem:catral}
A different generalisation is studied in~\cite{CS}: there the
roughness condition is replaced by $\lpf(m)\ge k\,g_n$, while the number
of rough integers permitted in a gap is not varied.
\end{remark}

\begin{remark}\label{rem:numerics-intro}
Section~\ref{sec:numerics} records numerical values of the local
constants. In particular, exact enumeration gives
$c=c_1\ge\sum_{h\le36}c_{h,1}=2.74136\ldots$, and for every
$3\le r\le 50$ the truncated Monte Carlo estimates
$\widehat c_r^{(\le900)}$ of $\sum_{h\le900}c_{h,r}$ given there satisfy
${|\widehat c_r^{(\le900)}-h_r^*|<3}$; this suggests that the error term
in~\eqref{eq:main-a} is far from optimal.
\end{remark}

\subsection{Notation}
The symbols $\ll$, $\gg$, $\asymp$ and $O(\cdot)$ have their usual meaning;
dependence of implied constants on parameters is indicated by
subscripts. The letters $\kappa,\kappa_1,\kappa_2,\dots$ denote positive
absolute constants (they are unrelated to the constants $c_r$). For even
$h$ we write $\PP$ and $\EE$ for probability and expectation with respect
to the uniform measure on $\mathbb Z/P_h\mathbb Z$, and regard $A_h$ as an
event; thus $\PP_{A_h}=\PP(\,\cdot\mid A_h)$.

\section{The quantitative transition}\label{sec:transition}

\subsection{Arithmetic inputs}

We use two classical estimates. The first is the average of the pair
singular series:
\begin{equation}\label{eq:singular-average}
\sum_{\substack{2\le h\le Y\\ h\ \mathrm{even}}}\Sigma(h)=Y+O(\log Y)
\qquad(Y\ge2).
\end{equation}
Indeed, for $h=2m$ one has $\Sigma(2m)=\fS_2\prod_{p>2,\,p\mid m}
\frac{p-1}{p-2}$, and \eqref{eq:singular-average} follows from the
estimate $\sum_{m\le w}\prod_{p>2,\,p\mid m}\frac{p-1}{p-2}
=2w/\fS_2+O(\log w)$ $(w>2)$ of Montgomery~\cite{Mo70}, in the form
recorded in \cite[(2.12)]{GT}, applied with $w=Y/2$.

The second is Mertens' product formula with the classical error term.

\begin{lemma}\label{lem:mertens}
There is an absolute constant $\kappa_1>0$ such that, for even $h\ge4$,
\[
V_h=\frac{e^{-\gamma}}{\log h}\Bigl(1+O\bigl(e^{-\kappa_1\sqrt{\log h}}\bigr)\Bigr).
\]
In particular $V_h\asymp1/\log h$, and $V_x/V_y=1+O(1/\log x)$ for even
$4\le x\le y\le 3x$.
\end{lemma}

\begin{proof}
Since $h$ is even and $h\ge4$, $h$ is not prime and
$V_h=\prod_{p\le h}(1-1/p)$. The estimate
$\prod_{p<z}(1-1/p)=e^{-\gamma}(\log z)^{-1}(1+O(\exp(-\kappa_1\sqrt{\log z})))$
follows from the prime number theorem with the classical error term; it
is stated in \cite[Lemma~3.1]{GT} (in the form
$R'=R(1+O(\exp(-\kappa_1\sqrt{\log z})))$), cf.\ \cite[(2.10)]{GT}. The last
assertion follows since $\log y/\log x=1+O(1/\log x)$ for $x\le y\le3x$.
\end{proof}

The key probabilistic input is the Montgomery--Vaughan bound for the
moments of reduced residues in short intervals.

\begin{lemma}\label{lem:MV}
For a modulus $q\ge1$ and an integer $L\ge1$ let
\[
m_K(q;L):=\sum_{a=1}^{q}\Bigl(\sum_{\substack{1\le m\le L\\(a+m,q)=1}}1
-L\frac{\varphi(q)}q\Bigr)^K .
\]
For every fixed positive integer $K$,
\begin{equation}\label{eq:MV}
\frac1q\,m_K(q;L)\ll_K\Bigl(L\frac{\varphi(q)}q\Bigr)^{K/2}
\Bigl(1+L^{-1/(7K)}\Bigl(\frac q{\varphi(q)}\Bigr)^{2^K+K/2}\Bigr).
\end{equation}
Consequently, if $K\ge2$ is a fixed even integer, $S_h(b):=\#\{0\le i\le h:
(b+i,P_h)=1\}$ and $\lambda_h:=(h+1)V_h$, then there is $h_0(K)$ such that
\begin{equation}\label{eq:moment-bound}
\EE|S_h-\lambda_h|^K\le C_K\,\lambda_h^{K/2}\qquad(h\ge h_0(K)\ \text{even}),
\end{equation}
where $C_K$ depends only on $K$.
\end{lemma}

\begin{proof}
By Lemma~2 of \cite{MV}, as recorded in \cite[(10), (11)]{MS},
$m_K(q;L)=q(\varphi(q)/q)^KV_K(q;L)$ for a certain exponential sum
$V_K(q;L)$, and Montgomery and Vaughan showed (see the display preceding
\cite[Theorem~1]{MS}) that
\[
V_K(q;L)\ll_K\Bigl(\frac{Lq}{\varphi(q)}\Bigr)^{K/2}
\Bigl(1+O_K\Bigl(L^{-1/(7K)}\Bigl(\frac q{\varphi(q)}\Bigr)^{2^K+K/2}\Bigr)\Bigr).
\]
Multiplying by $(\varphi(q)/q)^K$ gives~\eqref{eq:MV}.
For~\eqref{eq:moment-bound} take $q=P_h$ and $L=h+1$. Then
$L\varphi(q)/q=\lambda_h$ and $q/\varphi(q)=V_h^{-1}\ll\log h$ by
Lemma~\ref{lem:mertens}, so the second factor in~\eqref{eq:MV} is
$O_K(1)$ for $h\ge h_0(K)$. Finally, with $a=b-1$ the inner count in
$m_K(q;L)$ equals $S_h(b)$, and $K$ is even, so
$\EE|S_h-\lambda_h|^K=q^{-1}m_K(q;L)$.
\end{proof}

\subsection{Removing the endpoint conditioning}

\begin{lemma}\label{lem:endpoint}
Let $h\ge4$ be even. Then
\begin{equation}\label{eq:ratio}
\frac{\Sigma(h)}{\PP(A_h)}=V_h^{-2}\prod_{p>h}\Bigl(1-\frac1{(p-1)^2}\Bigr)
\le V_h^{-2},
\end{equation}
and on $A_h$ one has $S_h=\rho_h+2$. Consequently, for every $r\ge1$,
\begin{align}
\Sigma(h)\,q_{h,r}&\le V_h^{-2}\,\PP(S_h\le r+1),\label{eq:lower-tail}\\
\Sigma(h)\,(1-q_{h,r})&\le V_h^{-2}\,\PP(S_h\ge r+2).\label{eq:upper-tail}
\end{align}
\end{lemma}

\begin{proof}
By the Chinese remainder theorem,
$\PP(A_h)=\prod_{p<h}(1-\nu_p(\{0,h\})/p)$. Comparing with the Euler
product of $\Sigma(h)$, the factors with $p<h$ leave
$\prod_{p<h}(1-1/p)^{-2}=V_h^{-2}$. For $p\ge h$ we have $p\nmid h$ (as
$h$ is even, $h\ge4$, so $h$ is not prime), hence $\nu_p(\{0,h\})=2$ and
the corresponding factor is
$(1-2/p)(1-1/p)^{-2}=1-(p-1)^{-2}<1$; this proves~\eqref{eq:ratio}.
If $b\in A_h$ then $b$ and $b+h$ are coprime to $P_h$, so
$S_h(b)=\rho_h(b)+2$. Hence
$q_{h,r}=\PP(A_h\cap\{S_h\le r+1\})/\PP(A_h)\le\PP(S_h\le r+1)/\PP(A_h)$,
and~\eqref{eq:lower-tail} follows from~\eqref{eq:ratio};
\eqref{eq:upper-tail} is proved in the same way.
\end{proof}

\begin{lemma}\label{lem:convergence}
Fix $r\ge1$. There is $h_1(r)$ such that $c_{h,r}\ll(\log h)^5/h^3$ for
all even $h\ge h_1(r)$, with an absolute implied constant. In particular
$c_r<\infty$.
\end{lemma}

\begin{proof}
By Lemma~\ref{lem:mertens}, $\lambda_h\gg h/\log h$, so there is $h_1(r)$
such that $\lambda_h\ge2(r+1)$ and $h\ge h_0(6)$ for $h\ge h_1(r)$. Then
$\{S_h\le r+1\}\subseteq\{|S_h-\lambda_h|\ge\lambda_h/2\}$, and Markov's
inequality with~\eqref{eq:moment-bound} for $K=6$ gives
$\PP(S_h\le r+1)\le2^6C_6\lambda_h^{-3}$. By~\eqref{eq:lower-tail},
$c_{h,r}\ll V_h^{-2}\lambda_h^{-3}\ll(\log h)^5/h^3$.
\end{proof}

\subsection{Monotonicity and spacing of the arithmetic centre}

\begin{lemma}\label{lem:spacing}
Put $\mu_h:=hV_h$. For even $h\ge4$,
\[
\mu_{h+2}-\mu_h=
\begin{cases}
2V_h, & h+1\ \text{composite},\\[2pt]
\dfrac{h}{h+1}V_h, & h+1\ \text{prime}.
\end{cases}
\]
In particular $\frac45V_h\le\mu_{h+2}-\mu_h\le2V_h$, so $(\mu_h)$ is
strictly increasing along even $h\ge4$, and for even $4\le h\le h'$,
\begin{equation}\label{eq:spacing-sum}
\tfrac25\,(h'-h)\,V_{h'}\le\mu_{h'}-\mu_h\le(h'-h)\,V_h .
\end{equation}
\end{lemma}

\begin{proof}
The primes $p$ with $h\le p<h+2$ are among $h,h+1$, and $h$ is not prime.
If $h+1$ is composite then $V_{h+2}=V_h$; if $h+1$ is prime then
$V_{h+2}=V_h(1-\frac1{h+1})$. The formula follows, and $h/(h+1)\ge4/5$.
Summing the $(h'-h)/2$ increments and using that $V$ is non-increasing
gives~\eqref{eq:spacing-sum}.
\end{proof}

Note that $\mu_2=2>\mu_4=4/3$, which is why $h_r^*$ is only considered
for $r\ge3$.

\subsection{\texorpdfstring{Proof of Theorem~\ref{thm:main}(a)}{Proof of Theorem 1.4(a)}}

It suffices to treat $0<\eps\le1$. Fix such $\eps$ and an even integer
$K\ge\max(6,2/\eps)$; implied constants may depend on $\eps$ and $K$,
hence only on $\eps$. Let $r\ge3$, and put
\[
H:=h_r^*,\qquad W:=r^{1/2}(\log r)^{1+\eps}.
\]
It suffices to prove~\eqref{eq:main-a} for $r\ge r_0(\eps)$, a
sufficiently large constant: for $3\le r<r_0(\eps)$ both $c_r$ and $h_r^*$
are finite, and the claim holds after enlarging the implied constant.

\emph{Step 0: the size of $H$.} Since $\mu_h<3$ for even $h\le14$, we
have $H\ge16$, and by minimality $\mu_{H-2}<r\le\mu_H$. By
Lemma~\ref{lem:spacing}, $\mu_H-\mu_{H-2}\le2V_{H-2}<1$, so
$r\le\mu_H<r+1$. Lemma~\ref{lem:mertens} then gives $H/\log H\asymp r$,
hence $H\asymp r\log r$ and $\log H\asymp\log r$. We assume $r\ge r_0(\eps)$
so large that $r<H/2$, $4\le W<H/4$, $H/2\ge h_0(K)$ and $r\ge h_0(K)$.

\emph{Step 1: reduction.} Applying~\eqref{eq:singular-average} with
$Y=H-2$ gives $\sum_{2\le h<H,\ h\ \mathrm{even}}\Sigma(h)=H+O(\log H)$.
Since the series defining $c_r$ converges (Lemma~\ref{lem:convergence}),
\begin{equation}\label{eq:reduction}
c_r-H=\sum_{\substack{h\ge2\\h\ \mathrm{even}}}\Sigma(h)
\bigl(q_{h,r}-\mathbf 1_{h<H}\bigr)+O(\log H).
\end{equation}
We split the even $h\ge2$ into the ranges
\begin{gather*}
I_0:\ h\le r,\qquad
I_1:\ r<h<H/2,\qquad
I_2:\ H/2\le h\le2H,\ |h-H|>W,\\
I_3:\ |h-H|\le W,\qquad
I_4:\ h>2H .
\end{gather*}
These ranges are disjoint and cover all even $h\ge2$, since $r<H/2$ and
$W<H/2$.

\emph{Range $I_0$.} Here $\rho_h\le h-1<r$, so $q_{h,r}=1=\mathbf 1_{h<H}$,
and the contribution vanishes.

\emph{Range $I_3$.} By~\eqref{eq:singular-average}, the contribution is at
most $\sum_{|h-H|\le W}\Sigma(h)=2W+O(\log H)$.

\emph{Range $I_2$.} By Lemma~\ref{lem:mertens} there is an absolute
$\kappa_2>0$ with $V_{2H}\ge\kappa_2/\log H$. Let $\kappa_3:=\kappa_2/5$.
If $h\in I_2$ and $h>H$, then by~\eqref{eq:spacing-sum},
\[
\mu_h-r\ge\mu_h-\mu_H\ge\tfrac25(h-H)V_{2H}\ge\kappa_3\,\frac{h-H}{\log H}.
\]
If $h\in I_2$ and $h<H$, then $H-2-h\ge(H-h)/2$ (as $H-h>W\ge4$), and
\[
r-\mu_h>\mu_{H-2}-\mu_h\ge\tfrac25(H-2-h)V_{2H}\ge\kappa_3\,\frac{H-h}{\log H}.
\]
Since $\lambda_h=\mu_h+V_h$ with $0<V_h\le1$, it follows that
\[
\lambda_h-(r+1)\ge\kappa_3\frac{h-H}{\log H}-1\ge\frac{\kappa_3}2\,\frac{h-H}{\log H}
\quad(h>H),
\qquad
(r+2)-\lambda_h\ge\kappa_3\,\frac{H-h}{\log H}\quad(h<H),
\]
where in the first case we used $(h-H)/\log H>W/\log H\ge2/\kappa_3$ for
$r\ge r_0(\eps)$. By~\eqref{eq:lower-tail} (for $h>H$) or
\eqref{eq:upper-tail} (for $h<H$), Markov's inequality
and~\eqref{eq:moment-bound},
\[
\Sigma(h)\bigl|q_{h,r}-\mathbf 1_{h<H}\bigr|
\le V_h^{-2}\,C_K\lambda_h^{K/2}\Bigl(\frac{2\log H}{\kappa_3|h-H|}\Bigr)^K .
\]
For $h\in I_2$ we have $V_h^{-1}\le V_{2H}^{-1}\ll\log r$ and, by
monotonicity, $\lambda_h\le\mu_{2H}+1=2HV_{2H}+1\le2\mu_H+1<2r+3$. Hence
\[
\Sigma(h)\bigl|q_{h,r}-\mathbf 1_{h<H}\bigr|\ll
\frac{r^{K/2}(\log r)^{K+2}}{|h-H|^K},
\]
and summing over the even $h$ with $|h-H|>W$,
\[
\sum_{h\in I_2}\Sigma(h)\bigl|q_{h,r}-\mathbf 1_{h<H}\bigr|
\ll\frac{r^{K/2}(\log r)^{K+2}}{W^{K-1}}
=r^{1/2}(\log r)^{3-(K-1)\eps}\le W,
\]
because $K\eps\ge2$.

\emph{Range $I_1$.} Let $h_1'$ be the largest even integer $<H/2$; then
$h_1'\ge4$ and $h_1'\ge H/2-2$. For $h\in I_1$, monotonicity and
Lemma~\ref{lem:mertens} (note $H-2\le3h_1'$) give
\[
\mu_h\le\mu_{h_1'}=h_1'V_{h_1'}\le\frac H2\,V_{H-2}\,
\frac{V_{h_1'}}{V_{H-2}}
=\frac{H}{2(H-2)}\,\mu_{H-2}\Bigl(1+O\Bigl(\frac1{\log H}\Bigr)\Bigr)
\le\Bigl(\frac12+O\Bigl(\frac1{\log r}\Bigr)\Bigr)r .
\]
Hence $\lambda_h\le\mu_h+1\le0.6\,r$ and $(r+2)-\lambda_h\ge0.4\,r$ for
$r\ge r_0(\eps)$. As $h>r\ge h_0(K)$, \eqref{eq:upper-tail}, Markov's
inequality and~\eqref{eq:moment-bound} give
\[
\Sigma(h)(1-q_{h,r})\le V_h^{-2}C_K\lambda_h^{K/2}(0.4\,r)^{-K}
\ll(\log r)^2r^{-K/2}.
\]
Summing over fewer than $H\ll r\log r$ values of $h$, the contribution of
$I_1$ is $\ll r^{1-K/2}(\log r)^3\ll1$.

\emph{Range $I_4$.} Let $h_2'$ be the least even integer $>2H$, so
$h_2'\le2H+2\le3H$. For $h\in I_4$, monotonicity and
Lemma~\ref{lem:mertens} give
\[
\mu_h\ge\mu_{h_2'}=h_2'V_{h_2'}\ge2H\,V_H\,\frac{V_{h_2'}}{V_H}
=2\mu_H\Bigl(1+O\Bigl(\frac1{\log H}\Bigr)\Bigr)
\ge\Bigl(2-O\Bigl(\frac1{\log r}\Bigr)\Bigr)r .
\]
Hence $r+1\le0.55\,\mu_h\le0.55\,\lambda_h$ for $r\ge r_0(\eps)$, i.e.
$\lambda_h-(r+1)\ge0.45\,\lambda_h$. By~\eqref{eq:lower-tail}, Markov's
inequality and~\eqref{eq:moment-bound},
\[
\Sigma(h)\,q_{h,r}\le V_h^{-2}C_K\lambda_h^{K/2}(0.45\lambda_h)^{-K}
\ll V_h^{-2}\lambda_h^{-K/2}\ll\frac{(\log h)^{2+K/2}}{h^{K/2}},
\]
and the contribution of $I_4$ is $\ll H^{1-K/2}(\log H)^{2+K/2}\ll1$.

\emph{Conclusion.} Inserting these bounds into~\eqref{eq:reduction} gives
$c_r-H=O(\log H)+O(W)\ll_\eps W$, which is~\eqref{eq:main-a}.
\qed

\subsection{\texorpdfstring{Proof of Theorem~\ref{thm:main}(b) and (c)}{Proof of Theorem 1.4(b) and (c)}}

Write $H:=h_r^*$ and $f(x):=e^{-\gamma}x/\log x$, so that $f(H_r)=r$. As
in the previous proof we may assume $r\ge r_0$ large. By minimality of
$H$ and Lemma~\ref{lem:spacing},
\[
0\le\mu_H-r<\mu_H-\mu_{H-2}\le2V_{H-2}\ll\frac1{\log H}.
\]
By Lemma~\ref{lem:mertens}, $\mu_H=f(H)\bigl(1+O(e^{-\kappa_1\sqrt{\log H}})\bigr)$,
hence (as $1/\log H\le r\,e^{-\kappa_1\sqrt{\log H}}$ for large $r$)
\[
f(H)=r+O\bigl(r\,e^{-\kappa_1\sqrt{\log H}}\bigr).
\]
The function $f$ is increasing on $(e,\infty)$ with
$f'(x)=e^{-\gamma}(\log x-1)/(\log x)^2$. Both $H$ and $H_r$ are
$\asymp r\log r$ (for $H_r$ this follows from $H_r=e^\gamma r\log H_r$),
so $f'\asymp1/\log r$ between them, and the mean value theorem gives
\[
|H-H_r|\ll(\log r)\,r\,e^{-\kappa_1\sqrt{\log H}}\ll H_r\,e^{-\kappa_4\sqrt{\log r}},
\qquad\kappa_4:=\kappa_1/\sqrt2,
\]
since $\log H\ge\frac12\log r$ for large $r$. Combining this with
\eqref{eq:main-a} for $\eps=1$ yields
\[
c_r=H_r\bigl(1+O(e^{-\kappa_4\sqrt{\log r}})\bigr)+O\bigl(r^{1/2}(\log r)^2\bigr),
\]
and $r^{1/2}(\log r)^2/H_r\ll r^{-1/2}\log r\ll e^{-\kappa_4\sqrt{\log r}}$.
This proves~\eqref{eq:main-b} with $\kappa=\kappa_4$ for $r\ge r_0$; for
$3\le r<r_0$ it holds after enlarging the implied constant.

For~\eqref{eq:main-c}, put $\ell:=\log r$, $a:=\log\ell+\gamma$ and
$t:=H_r/(e^\gamma r)$. The defining equation gives $t=\log H_r>1$ and
\[
t=\ell+\gamma+\log t .
\]
The function $g(t):=t-\ell-\gamma-\log t$ is increasing on $(1,\infty)$,
so $t$ is its unique zero there. For $t_0:=\ell+a+a/\ell$, using
$\log(1+u)=u+O(u^2)$ with $u=a/\ell+a/\ell^2$ and $a\ge1$,
\[
g(t_0)=\frac a\ell-\log\Bigl(1+\frac a\ell+\frac a{\ell^2}\Bigr)
=O\Bigl(\frac{a^2}{\ell^2}\Bigr).
\]
Since $g'(t)=1-1/t\ge\frac12$ on $[t_0-1,t_0+1]$ and $g(t_0)\to0$, the
zero $t$ lies in this interval and $|t-t_0|\le2|g(t_0)|\ll a^2/\ell^2$.
Thus
\[
H_r=e^\gamma r\Bigl(\ell+\log\ell+\gamma+\frac{\log\ell+\gamma}{\ell}
+O\Bigl(\frac{(\log\ell)^2}{\ell^2}\Bigr)\Bigr).
\]
By~\eqref{eq:main-b}, $c_r-H_r\ll H_re^{-\kappa\sqrt\ell}\ll
r\,\ell\,e^{-\kappa\sqrt\ell}\ll r(\log\ell)^2/\ell^2$, which
gives~\eqref{eq:main-c}.
\qed

\begin{remark}\label{rem:mertens-cost}
Part (a) is more precise than (b): the passage from $h_r^*$ to $H_r$
costs only the error term in Mertens' product formula, and any
improvement of Lemma~\ref{lem:mertens} (for instance via the
Vinogradov--Korobov zero-free region) improves~\eqref{eq:main-b}
accordingly.
\end{remark}

\section{Prime gaps with fewer than \texorpdfstring{$r$}{r} rough numbers}\label{sec:gaps}

Throughout this section $r\ge1$ is fixed and $X$ is large. For $X\ge3$
every $g_n$ with $p_n\ge X$ is even.

\subsection{Long gaps}

\begin{lemma}\label{lem:long}
Let $0<\eta\le1$ and $B>0$ be fixed. Then
\[
\#\{n:\ X\le p_n\le2X,\ g_n\ge(\log X)^\eta,\ M_n<r\}
\ll_{r,\eta,B}\frac X{(\log X)^B}.
\]
\end{lemma}

\begin{proof}
Fix $\alpha,\beta$ with $0<\alpha<\eta$ and $0<\beta<\alpha/2$, and a
fixed even integer $K$ with $K(\alpha/2-\beta)+\alpha>B$. Put
\[
L:=\lfloor(\log X)^\alpha\rfloor,\qquad z:=\exp((\log X)^\beta),\qquad
R':=L\prod_{p<z}\Bigl(1-\frac1p\Bigr),
\]
and $T(x):=\#\{x<m\le x+L:\ \lpf(m)\ge z\}$ for real $x$. Since
$0<\beta<\alpha<1$, \cite[Lemma~3.1]{GT} (with $H=L$) gives, for even $K$,
\begin{equation}\label{eq:GT-moment}
\int_X^{2X}|T(x)-R'|^K\,dx\ll_K X\,L^{K/2+o(1)} .
\end{equation}
By Mertens' theorem $R'\asymp(\log X)^{\alpha-\beta}$, so $R'\ge2r$ for
large $X$, and the set $\mathcal E:=\{x\in[X,2X]:\ T(x)<r\}$ satisfies,
by~\eqref{eq:GT-moment},
\[
|\mathcal E|\le\Bigl(\frac{R'}2\Bigr)^{-K}\int_X^{2X}|T(x)-R'|^K\,dx
\ll X(\log X)^{-K(\alpha/2-\beta)+o(1)} .
\]
Let $n$ be counted in the lemma. At most one such $n$ has $p_{n+1}>2X$.
If $p_{n+1}\le2X$ and $g_n\ge z$, then $(p_n,p_{n+1})$ is one of at most
$X/z$ disjoint subintervals of $[X,2X]$ of length $\ge z$, and
$X/z\ll X(\log X)^{-B}$. Finally suppose $p_{n+1}\le2X$ and
$(\log X)^\eta\le g_n<z$. For every real $x\in[p_n,p_{n+1}-L)$ all integers
$m\in(x,x+L]$ satisfy $p_n<m<p_{n+1}$, and every such $m$ with
$\lpf(m)\ge z$ has $\lpf(m)>g_n$, i.e.\ is rough for the gap. Hence
$T(x)\le M_n<r$, so $[p_n,p_{n+1}-L)\subseteq\mathcal E$. These intervals
are pairwise disjoint and have length $g_n-L\ge L$ (as
$(\log X)^\eta\ge2L$ for large $X$). Therefore the number of such $n$ is
at most $|\mathcal E|/L\ll X(\log X)^{-K(\alpha/2-\beta)-\alpha+o(1)}
\ll X(\log X)^{-B}$.
\end{proof}

\subsection{Short gaps}

\begin{lemma}\label{lem:short}
There is $h_2(r)$ such that for all sufficiently large $X$ and all even
$h$ with $h_2(r)\le h\le(\log X)^{1/100}$,
\[
N_{h,r}(X)\ll\frac{X}{h^2\log^2X},
\]
with an absolute implied constant.
\end{lemma}

\begin{proof}
Put $L:=h-1$, $z:=L^{10}$, $Q:=\prod_{p<z}p$,
\[
U(b):=\#\{1\le i\le L:\ (b+i,Q)=1\},\qquad R':=L\,\frac{\varphi(Q)}Q .
\]
As $h<z$, $(m,Q)=1$ implies $\lpf(m)\ge z>h$. Let $n$ be counted by
$N_{h,r}(X)$ and $b\equiv p_n\pmod Q$. Each $i$ counted by $U(b)$ gives a
rough integer $p_n+i$ of the gap, so $U(b)\le M_n<r$; moreover
$(b(b+h),Q)=1$, since $p_n$ and $p_n+h$ are primes exceeding~$z$.

By~\eqref{eq:MV} with $q=Q$ and $K=6$, and since
$Q/\varphi(Q)=\prod_{p<z}(1-1/p)^{-1}\ll\log L$,
\[
\frac1Q\sum_{b\bmod Q}\bigl(U(b)-R'\bigr)^6
\ll R'^3\bigl(1+L^{-1/42}(\log L)^{67}\bigr)\ll R'^3 .
\]
Since $R'\asymp L/\log L$, we may choose $h_2(r)$ so that $R'\ge2r$; then
$U(b)<r$ implies $|U(b)-R'|\ge R'/2$, and
\[
\#\{b\bmod Q:\ U(b)<r\}\ll Q\,R'^{-3} .
\]
For each $b\bmod Q$ with $(b(b+h),Q)=1$, the upper-bound sieve estimate
used in \cite[proof of Lemma~4.1]{GT} gives
\[
\#\{X\le p\le2X:\ p\equiv b\ (\mathrm{mod}\ Q),\ p+h\ \text{prime}\}
\ll L^{o(1)}\,\frac{X}{Q\log^2X}\qquad(L\to\infty),
\]
uniformly in such $b$ and in $h\le(\log X)^{1/100}$; here
$\log Q\ll z\le(\log X)^{1/10}$, so $Q=X^{o(1)}$. Summing over the
admissible classes with $U(b)<r$,
\[
N_{h,r}(X)\ll Q\,R'^{-3}\,L^{o(1)}\,\frac{X}{Q\log^2X}
=L^{-3+o(1)}\,\frac{X}{\log^2X}\ll\frac{X}{h^2\log^2X},
\]
after enlarging $h_2(r)$ if necessary.
\end{proof}

\subsection{\texorpdfstring{Proof of Theorem~\ref{thm:gaps}(a)}{Proof of Theorem 1.5(a)}}

Let $n$ be counted by $N_r(X)$ and $h:=g_n$. If $h<h_2(r)$, we use
$N_{h,r}(X)\le\#\{n:\ X\le p_n\le2X,\ g_n=h\}\ll\Sigma(h)X/\log^2X$
(\cite[(2.14)]{GT}); there are $O_r(1)$ such~$h$. The even
$h\in[h_2(r),(\log X)^{1/100}]$ contribute
$\ll\sum_hX/(h^2\log^2X)\ll X/\log^2X$ by Lemma~\ref{lem:short}, and the
gaps with $g_n\ge(\log X)^{1/100}$ contribute $\ll_rX/\log^3X$ by
Lemma~\ref{lem:long} with $\eta=1/100$ and $B=3$. Hence
$N_r(X)\ll_rX/\log^2X$.
\qed

\subsection{The conditional asymptotic}

We use the following special case of the Dickson--Hardy--Littlewood
prime tuples conjecture~\cite{Di04,HL}.

\begin{hypothesis}\label{hyp:DHL}
Let $W\ge1$ and let $b$ and $h\ne0$ be integers such that $b$ and $b+h$
are coprime to $W$, all fixed. Then, as $X\to\infty$,
\[
\#\{X\le p\le2X:\ p\equiv b\ (\mathrm{mod}\ W),\ p+h\ \text{prime}\}
=\bigl(\fS_W(b,h)+o(1)\bigr)\frac{X}{W\log^2X},
\]
where
\[
\fS_W(b,h):=\prod_{p\mid W}\Bigl(1-\frac1p\Bigr)^{-2}
\prod_{p\nmid W}\Bigl(1-\frac1p\Bigr)^{-2}\Bigl(1-\frac{\nu_p(\{0,h\})}p\Bigr).
\]
\end{hypothesis}

This is \cite[Conjecture~4.1]{GT} (with $b_1=b$, $b_2=b+h$) in the
normalisation in which it is applied in \cite[Section~4]{GT}: the progression
contains $\sim X/W$ integers of $[X,2X]$, and for $p\nmid W$ the classes
$b,b+h$ occupy $\nu_p(\{0,h\})$ residues modulo~$p$.

\begin{proposition}\label{prop:fixed-h}
Assume Hypothesis~\ref{hyp:DHL}. Then for every fixed even $h\ge2$,
\[
N_{h,r}(X)=\bigl(c_{h,r}+o(1)\bigr)\frac{X}{\log^2X}\qquad(X\to\infty).
\]
\end{proposition}

\begin{proof}
Let $B_{h,r}:=\{b\in A_h:\ \rho_h(b)<r\}$, so $|B_{h,r}|=q_{h,r}|A_h|$.
For $b\in A_h$ and $W=P_h$, the factors of $\fS_{P_h}(b,h)$ with $p<h$
are $(1-1/p)^{-2}$, and those with $p\ge h$ coincide with the
corresponding factors of $\Sigma(h)$. Since
$|A_h|=P_h\prod_{p<h}(1-\nu_p(\{0,h\})/p)$, this gives
$\fS_{P_h}(b,h)=\Sigma(h)P_h/|A_h|$. Hence Hypothesis~\ref{hyp:DHL}
yields, for each $b\in A_h$,
\[
\#\{X\le p\le2X:\ p\equiv b\ (\mathrm{mod}\ P_h),\ p+h\ \text{prime}\}
=\Bigl(\frac{\Sigma(h)}{|A_h|}+o(1)\Bigr)\frac X{\log^2X},
\]
and summing over $b\in B_{h,r}$,
\[
\mathcal A_{h,r}(X):=\#\{X\le p\le2X:\ p+h\ \text{prime},\
p\bmod P_h\in B_{h,r}\}=\bigl(c_{h,r}+o(1)\bigr)\frac X{\log^2X}.
\]
Let $X>h$. By Lemma~\ref{lem:interpretation}, $n\mapsto p_n$ maps the
set counted by $N_{h,r}(X)$ injectively into the set counted by
$\mathcal A_{h,r}(X)$. Conversely, if $p$ is counted by
$\mathcal A_{h,r}(X)$ and $p+h$ is the prime following $p$, then $p=p_n$
with $g_n=h$ and $M_n=\rho_h(p)<r$ (Lemma~\ref{lem:interpretation}).
Hence
\[
0\le\mathcal A_{h,r}(X)-N_{h,r}(X)\le\sum_{i=1}^{h-1}
\#\{X\le p\le2X:\ p,\ p+i,\ p+h\ \text{all prime}\}.
\]
For fixed $h$ and $i$ each term is $O_h(X/\log^3X)$ by an upper-bound
sieve of dimension three (see e.g.~\cite{FI}), and it is $O_h(1)$ if
$\{0,i,h\}$ is not admissible. Thus
$N_{h,r}(X)=\mathcal A_{h,r}(X)+O_h(X/\log^3X)$, which proves the claim.
\end{proof}

For $r=1$ one has $\rho_h(b)=0$ for $b\in B_{h,1}$, so no prime can lie
between $p$ and $p+h$ and $\mathcal A_{h,1}(X)=N_{h,1}(X)$ for $X>h$; in
this case Proposition~\ref{prop:fixed-h} is \cite[(4.2)]{GT}.

\subsection{\texorpdfstring{Proof of Theorem~\ref{thm:gaps}(b)}{Proof of Theorem 1.5(b)}}

Fix an even $Y\ge h_2(r)$. For large $X$, Lemmas~\ref{lem:long}
(with $\eta=1/100$, $B=3$) and~\ref{lem:short} give
\[
N_r(X)=\sum_{\substack{h\le Y\\h\ \mathrm{even}}}N_{h,r}(X)
+O\Bigl(\frac{X}{Y\log^2X}\Bigr)+O_r\Bigl(\frac X{\log^3X}\Bigr),
\]
where the first implied constant is absolute. By
Proposition~\ref{prop:fixed-h},
\[
\limsup_{X\to\infty}\Bigl|\frac{N_r(X)\log^2X}{X}
-\sum_{\substack{h\le Y\\h\ \mathrm{even}}}c_{h,r}\Bigr|\ll\frac1Y .
\]
Letting $Y\to\infty$ and using Lemma~\ref{lem:convergence}, we obtain
$N_r(X)\log^2X/X\to c_r$. Since $c_r\ge\fS_2>0$, this is
$N_r(X)\sim c_rX/\log^2X$.
\qed

\section{Numerical illustration}\label{sec:numerics}

The computations in this section concern only the local model; they are
not used in any of the proofs above.

\subsection{The case \texorpdfstring{$r=1$}{r=1}}
By the Chinese remainder theorem, $A_h$ corresponds to the set of tuples
$(b_p)_{p<h}$ with $b_p\not\equiv0,-h\pmod p$, and $\rho_h(b)$ is the
number of $1\le i\le h-1$ with $b_p+i\not\equiv0\pmod p$ for all $p<h$.
Enumerating all $|A_h|$ tuples, we computed $|\Omega_h|$ exactly for even
$h\le36$; see Table~\ref{tab:r1}, where $c_{h,1}$ is obtained
from~\eqref{eq:Sigma-over-A}. For $h\le18$ our values of $|\Omega_h|$ agree with
\cite[Table~1]{GT}.\footnote{Two entries of \cite[Table~1]{GT} appear to
be misprints. For $h=12$ the list given there contains
$2310\equiv0\pmod{2310}$, whereas
$\Omega_{12}=\{1,199,467,509,1789,1831,2099,2297\}$ has $8$ elements, in
agreement with the value $c_{12}=0.0782$ given there. For $h=14$ the
entry $0.0229$ is inconsistent with the cumulative column
($2.6824-2.6308=0.0516$), which agrees with our value.} Since all terms
of $c_1$ are nonnegative and $c_1=c$ (Remark~\ref{rem:GT}),
\[
c\ \ge\ \sum_{\substack{h\le36\\ h\ \mathrm{even}}}c_{h,1}
=\fS_2\cdot\frac{12928101827}{6226553025}=2.7413688\ldots .
\]
A Monte Carlo estimate of the terms with $38\le h\le900$ (independent
uniform residues $b_p$; $10^6$ samples for each $h\le100$ and
$2\cdot10^4$ samples for $100<h\le900$) adds about $(2.1\pm0.2)\cdot10^{-4}$
(one standard deviation). In these samples $\rho_h=0$ occurred only for
$h\le68$, and it did not occur in the supplementary samples for
$902\le h\le2000$ described in Section~\ref{subsec:larger-r}. This suggests $c\approx2.7416$; this value
is not rigorous.

\begin{table}[ht]
\centering
\small
\begin{tabular}{rrrrr}
\toprule
$h$ & $|A_h|$ & $|\Omega_h|$ & $c_{h,1}$ & $\sum_{h'\le h}c_{h',1}$\\
\midrule
4 & 1 & 1 & 1.320324 & 1.320324\\
6 & 6 & 2 & 0.880216 & 2.200539\\
8 & 15 & 2 & 0.176043 & 2.376583\\
10 & 20 & 2 & 0.176043 & 2.552626\\
12 & 270 & 8 & 0.078241 & 2.630867\\
14 & 1782 & 58 & 0.051568 & 2.682435\\
16 & 1485 & 12 & 0.010669 & 2.693105\\
18 & 44550 & 376 & 0.022287 & 2.715392\\
20 & 504900 & 1440 & 0.005021 & 2.720412\\
22 & 420750 & 2622 & 0.009142 & 2.729554\\
24 & 15904350 & 48868 & 0.008114 & 2.737668\\
26 & 8675100 & 7682 & 0.001275 & 2.738944\\
28 & 9542610 & 5664 & 0.000940 & 2.739884\\
30 & 572556600 & 154220 & 0.000948 & 2.740832\\
32 & 6226553025 & 833688 & 0.000177 & 2.741009\\
34 & 6641656560 & 1075458 & 0.000228 & 2.741237\\
36 & 12453106050 & 620632 & 0.000132 & 2.741369\\
\bottomrule
\end{tabular}
\caption{Exact values of $|\Omega_h|$ and $c_{h,1}$ (rounded to six
decimals); $c_{2,1}=0$.}
\label{tab:r1}
\end{table}

\subsection{Larger \texorpdfstring{$r$}{r}}\label{subsec:larger-r}
Put
\[
c_r^{(\le900)}:=\sum_{\substack{2\le h\le900\\ h\ \mathrm{even}}}c_{h,r};
\qquad\text{thus}\qquad c_r^{(\le900)}\le c_r ,
\]
since every $c_{h,r}$ is nonnegative. Table~\ref{tab:rgrow} lists Monte
Carlo estimates $\widehat c_r^{(\le900)}$ of the truncated sums
$c_r^{(\le900)}$ and compares them with $h_r^*$ and $H_r$. Here the exact
distributions of $\rho_h$ were used for $h\le36$ and the Monte Carlo
samples described above for $38\le h\le900$; the column $\sigma$ is one
standard deviation of the Monte Carlo error. The terms with $h>900$ are
not included, and $\sigma$ does not account for them. To gauge their
size we drew $10^5$ further samples of $\rho_h$ for each even $h$ with
$902\le h\le2000$. The smallest value observed was $\rho_h=50$, so these
$h$ contribute nothing to the estimates for $r\le50$ (for comparison, the
Monte Carlo estimates of $\sum_{902\le h\le2000}c_{h,r}$ from the same
samples are about $1.5\cdot10^{-3}$ for $r=55$ and $6\cdot10^{-2}$ for
$r=60$). We therefore expect $c_r-c_r^{(\le900)}$ to be negligible at the
precision of Table~\ref{tab:rgrow} for $r\le50$, but we have not proved
this.

\begin{table}[ht]
\centering
\small
\begin{tabular}{rrrrrr}
\toprule
$r$ & $\widehat c_r^{(\le900)}$ & $\sigma$ & $h_r^*$ & $H_r$ &
$\widehat c_r^{(\le900)}-h_r^*$\\
\midrule
3 & 17.61 & $<0.01$ & 16 & 14.16 & $+1.61$\\
4 & 25.51 & $<0.01$ & 26 & 22.03 & $-0.49$\\
5 & 33.88 & $<0.01$ & 34 & 30.41 & $-0.12$\\
6 & 42.78 & $<0.01$ & 42 & 39.21 & $+0.78$\\
8 & 61.62 & 0.01 & 60 & 57.81 & $+1.62$\\
10 & 81.59 & 0.02 & 82 & 77.48 & $-0.41$\\
12 & 102.29 & 0.03 & 100 & 97.99 & $+2.29$\\
15 & 134.35 & 0.04 & 134 & 130.05 & $+0.35$\\
20 & 190.76 & 0.04 & 190 & 186.18 & $+0.76$\\
25 & 249.95 & 0.05 & 250 & 244.94 & $-0.05$\\
30 & 310.85 & 0.05 & 310 & 305.79 & $+0.85$\\
40 & 437.71 & 0.06 & 436 & 432.40 & $+1.71$\\
50 & 569.50 & 0.06 & 568 & 564.19 & $+1.50$\\
\bottomrule
\end{tabular}
\caption{Truncated Monte Carlo estimates $\widehat c_r^{(\le900)}$ of
$c_r^{(\le900)}=\sum_{h\le900}c_{h,r}$ (a lower bound for $c_r$), compared
with $h_r^*$ and $H_r$.}
\label{tab:rgrow}
\end{table}

For every $3\le r\le50$ (not only for the values listed) the truncated
estimates satisfy ${|\widehat c_r^{(\le900)}-h_r^*|<3}$, with $\sigma<0.06$;
the extreme values are $-1.66$ at $r=21$ and $+2.75$ at $r=48$. These
deviations show no tendency to grow, whereas $r^{1/2}\log r$ increases
from about $1.9$ to about $27.7$ over this range; this suggests that
Theorem~\ref{thm:main}(a) is far from optimal. Part of the fluctuation
reflects the fact that $h_r^*$ runs through even integers only. A sharper estimate would require a local limit theorem for
$\rho_h$ under $\PP_{A_h}$, uniformly in $h$ near $h_r^*$, which we do not
pursue here.

\FloatBarrier

\end{document}